\def\draftdate{\today}

\documentclass{amsart}
\usepackage{amssymb}
\usepackage{xy}
\usepackage{bm}

\newcommand{\overto}[1]{\xrightarrow{\,#1\,}}

\let\Lim\lim
\def\lim{\Lim\nolimits}

\mathchardef\varDelta="7101

\newcommand{\phat}[1][{p}]{^{\wedge}_{#1}}

\let\iso\cong
\let\sma\wedge
\newcommand{\htp}{\simeq}
\renewcommand{\to}{\mathchoice{\longrightarrow}{\rightarrow}{\rightarrow}{\rightarrow}}

\DeclareMathAlphabet{\catsymbfont}{U}{rsfs}{m}{n}

\newcommand{\aF}{{\catsymbfont{F}}}

\newcommand{\oH}{\mathcal{H}}

\newcommand{\bN}{{\mathbb{N}}}
\newcommand{\bS}{{\mathbb{S}}}

\newcommand{\bR}{{\mathbb{R}}}
\newcommand{\bT}{{\mathbb{T}}}
\newcommand{\bZ}{{\mathbb{Z}}}

\newcommand{\dL}{\mathbf{L}}

\def\quickop#1{\expandafter\DeclareMathOperator\csname
#1\endcsname{#1}}
\quickop{id}\quickop{Id}\quickop{holim}\quickop{hocolim}\quickop{op}
\quickop{co}\quickop{Ar}\quickop{sing}\quickop{Hom}\quickop{w}\quickop{Ho}
\quickop{ob}\quickop{diag}\quickop{Cyc}\quickop{Fib}\quickop{Cof}
\quickop{tr}\quickop{trc}\quickop{Gal}\quickop{colim}\quickop{triv}
\quickop{red}\quickop{hoeq}\quickop{Sp}\quickop{cyc}\quickop{can}

\numberwithin{equation}{section}

\newtheorem{thm}[equation]{Theorem}
\newtheorem*{thm*}{Theorem}
\newtheorem{cor}[equation]{Corollary}

\newtheorem{prop}[equation]{Proposition}

\theoremstyle{definition}
\newtheorem{defn}[equation]{Definition}

\newtheorem{conv}[equation]{Convention}
\theoremstyle{remark}

\xyoption{arrow}
\xyoption{curve}
\xyoption{matrix}
\xyoption{cmtip}
\SelectTips{cm}{}

\newcommand{\term}[1]{\textit{#1}}

\newdir{ >}{{}*!/-5pt/\dir{>}}

\begin{document}

\title[Chromatic convergence for $TC$]%
{A version of Waldhausen's chromatic convergence for $TC$}

\author{Andrew J. Blumberg}
\address{Department of Mathematics, Columbia University, 
New York, NY \ 10027}
\email{blumberg@math.columbia.edu}
\thanks{The first author was supported in part by NSF grant DMS-1812064}
\author{Michael A. Mandell}
\address{Department of Mathematics, Indiana University,
Bloomington, IN \ 47405}
\email{mmandell@indiana.edu}
\thanks{The second author was supported in part by NSF grant DMS-1811820}
\author{Allen Yuan}
\address{Department of Mathematics, Columbia University,
New York, NY \ 10027}
\email{yuan@math.columbia.edu}
\thanks{The third author was supported in part by NSF grant DMS-2002029}

\date{\draftdate}

\begin{abstract}
The map $TC(\mathbb{S})^{\wedge}_{p}\to \holim
TC(L_{n}\mathbb{S})^{\wedge}_{p}$ is a weak equivalence. 
\end{abstract}

\maketitle


\section{Introduction}

Motivated by the chromatic program in stable homotopy theory in
general and chromatic convergence of the sphere spectrum in
particular, Waldhausen~\cite{Waldhausen-Chromatic} proposed studying
the interaction of the chromatic tower with algebraic $K$-theory and
made some specific conjectures about what happens for the $p$-local sphere
spectrum.  Waldhausen conjectured that $K(\bS_{(p)})$ would be the homotopy
limit of the algebraic $K$-theory spectra of $L^{f}_n \bS$, which
would inductively be built from the algebraic $K$-theory of the
monochromatic categories. The monochromatic categories are simpler
than the $p$-local stable category and Waldhausen asked whether their algebraic
$K$-theories would be correspondingly simpler to understand than the
algebraic $K$-theory of the $p$-local sphere spectrum.

The advent of trace methods and the pioneering work of McClure and
Staffeldt in calculating $THH$ of the Adams summand $\ell$
reinvigorated this program.  Ausoni and Rognes~\cite{AusoniRognes}
calculated the periodic homotopy groups of $K(\ell)$ in terms of
$TC(\ell)$.  As part of this work, Rognes made a series of far-reaching conjectures about the
interaction of algebraic $K$-theory and chromatic localization.
Most notably, he formulated a redshift conjecture about 
$K$-theory increasing chromatic complexity, and he formulated a higher chromatic
Quillen-Lichtenbaum conjecture about Galois descent for algebraic $K$-theory
spectra.  Rognes' original formulations involved the localizations
$L_{n}$ (in place of $L_{n}^{f}$), the $p$-complete
sphere, and a $p$-complete version of Waldhausen's chromatic
convergence conjecture, $K(\bS\phat) \htp \holim_n K(L_n \bS\phat)$.

When working with the $p$-complete sphere, there is very little
difference between algebraic $K$-theory and $TC$: work of
Hesselholt-Madsen~\cite[Thm.~D]{HM2} and Dundas~\cite{Dundas-RelK} shows
that the cyclotomic trace $K(\bS\phat)\phat\to TC(\bS\phat)\phat$ is a
weak equivalence on connective covers.  As a consequence,
one is led to consider the $TC$ analogues of Waldhausen's chromatic
convergence conjectures. The purpose of this paper is to prove that
analogue. 

\begin{thm}\label{main}
The natural map
\[
TC(\bS\phat)\phat \to \holim_n TC(L_n (\bS\phat))\phat
\]
is a weak equivalence.
\end{thm}

The ring spectra $L_{n}(\bS\phat)$ are non-connective, and for
non-connective ring spectra, we now have various potentially distinct
versions of $TC(-)\phat$.  Theorem~\ref{main} holds using either the
original constructions $TC(-)\phat$ and $TC(-;p)\phat$ of
Goodwillie~\cite{Goodwillie-MSRI} and
B\"okstedt-Hsiang-Madsen~\cite{BHM} (which are equivalent for all ring
spectra), or the constructions $TC(-)\phat$ and $TC(-;p)\phat$ of
Nikolaus-Scholze~\cite[II.1.8]{NikolausScholze}.

McClure and Staffeldt~\cite{McClureStaffeldt-Chromatic} proved a
version of Waldhausen's chromatic convergence conjecture for the
connective covers $\tau_{\geq 0} L_n \bS$ in place of $L_{n}\bS$ using
a direct calculation involving the plus construction.  Their argument
adapts in outline to $TC$ and we have a corresponding result in this case.

\begin{thm}\label{main:conn}
The natural map
\[
TC(\bS\phat)\phat \to \holim_n TC(\tau_{\geq 0}L_n (\bS\phat))\phat
\]
is a weak equivalence.
\end{thm}

One might hope to deduce the previous theorem directly from the
McClure-Staffeldt $K$-theory result using the Dundas-Goodwillie-McCarthy
pullback square.  However, our current state of knowledge of
$\pi_{0}L_{n}\bS$ makes this infeasible, and so we give a direct proof
along the lines of the McClure-Staffeldt argument instead.
The McClure-Staffeldt argument for Theorem~\ref{main:conn}
fundamentally relies on the Chromatic Convergence Theorem of Hopkins
and Ravenel~\cite[7.5.7]{Ravenel-Nilpotence} applied to $\bS$ or $\bS\phat$,
propagating it through the construction of $TC$.  Whereas in the
$K$-theory case considered by McClure-Staffeldt, work of
Bousfield-Kan~\cite[III\S3]{BousfieldKan} on pro-spaces suffices for the
propagation, the $TC$ case needs the more sophisticated theory of
Isaksen~\cite{Isaksen-ProSpectra} and
Fausk-Isaksen~\cite{FauskIsaksen} for pro-spectra and equivariant
pro-spectra. (See Section~\ref{sec:conn}.) 
The argument for Theorem~\ref{main} fits a similar outline but uses
the weak equivalence $THH(L_{n}\bS)\simeq L_{n}\bS$ to avoid
pro-object arguments and requires chromatic convergence at a later
step: it replaces the Hopkins-Ravenel Chromatic Convergence Theorem
with Barthel's criterion for chromatic
convergence~\cite[3.8]{Barthel-ChromaticConvergence} applied to
$\Sigma^{\infty}_{+} BC_{p^{k}}$.  (See Section~\ref{sec:nonconn}.)  Although not
directly related to Waldhausen's conjecture or the Rognes program,
Barthel's criterion also gives the following chromatic convergence
result, which is superficially similar to Theorem~\ref{main} and also
proved in Section~\ref{sec:nonconn}.

\begin{thm}\label{main:cc}
The natural map
\[
TC(\bS\phat)\phat \to \holim_n (L_{n}TC(\bS\phat))\phat
\]
is a weak equivalence.
\end{thm}

We have stated Theorems~\ref{main}, \ref{main:conn}, and~\ref{main:cc}
in terms of $\bS\phat$, but the corresponding results with $\bS\phat$
replaced everywhere by $\bS$ or $\bS_{(p)}$ also hold; see
Theorem~\ref{thm:TCS}.  Indeed, the map $TC(R)\phat\to
TC(R\phat)\phat$ is a weak equivalence quite generally; we give precise statements in Propositions~\ref{prop:classicp},
\ref{prop:NSp}, and~\ref{prop:pcompl}.  In the proofs, we concentrate on the
case of $\bS_{(p)}$, which implies the other cases.

\subsection*{Conventions}
In this paper, when a point-set category of spectra is necessary or
implicit in a statement, we understand ``spectrum'' to mean orthogonal
spectrum indexed on $\{\bR^{n}\}$;  ``ring spectrum'' means
associative ring orthogonal spectrum (or the equivalent) and
``commutative ring spectrum'' means commutative ring orthogonal
spectrum (or the equivalent).  We use $\bT$ to denote the circle group
(the unit length complex numbers) and $C_{m}<\bT$ its cyclic subgroup
of order $m$.

\subsection*{Acknowledgments}
The authors thank Mike Hopkins and John Rognes for helpful conversations.

\section{$TC$ of a ring spectrum and its $p$-completion}
\label{sec:review}

In this section, we give a concise review of the construction of
topological cyclic homology in order to introduce notation and details
used in later sections. 
We discuss both the ``classic'' definition of a cyclotomic spectrum as
introduced by Goodwillie~\cite{Goodwillie-MSRI} and
B\"okstedt-Hsiang-Madsen~\cite{BHM} as well as the recent reformulation
by Nikolaus-Scholze in terms of the Tate fixed point spectrum.

In the classic case, definitions of point-set categories of cyclotomic
spectra and/or $p$-cyclotomic spectra can be found for example
in~\cite[1.1]{BokstedtMadsen-TCZ} (original source),
\cite[2.4.4]{Madsen-Traces}, \cite[2.2]{HM2}, \cite[\S1.1]{HMAnnals},
to name a few.  The notion of the homotopy theory and therefore
$\infty$-category of classic cyclotomic spectra and $p$-cyclotomic
spectra are clear but implicit in these formulations; the
paper~\cite{BM-cycl} sets up a model* category of $p$-cyclotomic
spectra for the purpose of making this explicit.  The details involve
the genuine $\bT$-equivariant stable homotopy theory localized at the
$\aF_{p}$-equivalences, meaning that a map $X\to Y$ is a weak
equivalence when it is a weak equivalence on (derived genuine
categorical) $C_{p^{k}}$-fixed points
\[
X^{C_{p^{k}}}\to Y^{C_{p^{k}}}
\]
for all $k\geq 0$.  (Here ``genuine'' means stable with respect all
finite dimensional orthogonal $\bT$-representations: for $V$ any
orthogonal $\bT$-action on $\bR^{n}$ for some $n$, the endofunctor
$\Sigma^{V}$ is a categorical equivalence.) The (derived) $C_{p}$
geometric fixed point functor $\Phi^{C_{p}}$ (or $X\mapsto X^{\Phi
C_{p}}$) goes from genuine $\aF_{p}$-local $\bT$-spectra
to genuine $\aF_{p}$-local $\bT/C_{p}$-spectra and the
$p$th root isomorphism $\rho_{p} \colon \bT\to \bT/C_{p}$ gives a
categorical equivalence $\rho^{*}_{p}$ from genuine $\aF_{p}$-local
$\bT/C_{p}$-spectra to genuine $\aF_{p}$-local
$\bT$-spectra.  The \term{structure map} of a classic
$p$-cyclotomic spectrum is a weak equivalence
\[
r\colon \rho^{*}_{p}X^{\Phi C_{p}}\to X.
\]
Construction~5.11 and Corollary~5.13 in \cite{BM-cycl} identify the
mapping spectrum in $p$-cyclotomic spectra $F_{\cyc}(X,Y)$ (for
$p$-cyclotomic spectra $X,Y$) and the underlying mapping space is the
space of homotopy commuting diagrams of structure maps; this in
particular identifies the $\infty$-category of $p$-cyclotomic spectra
as the full subcategory of the lax equalizer (in the sense
of~\cite[II.1.4]{NikolausScholze}) from the $C_{p}$ geometric fixed
point to the identity functors of those objects for which the map is
an isomorphism.   

Barwick-Glasman~\cite{BarwickGlasman-Cyclonic} sets up an equivalent
$\infty$-category of $p$-cyclotomic spectra in a different
$\infty$-categorical framework using a spectral Mackey functor
inspired model for equivariant stable homotopy theory.  This theory
contains features equivalent to those discussed in the previous
paragraph (cf.\ the proof of Theorem~3.23 there); the remainder of this
section and Section~\ref{sec:nonconn} adapt to the
Barwick-Glasman~\cite{BarwickGlasman-Cyclonic} context without
difficulty.  We expect that the arguments in Section~\ref{sec:conn}
also adapt either by choosing specific point-set models or by adapting
the model theoretic arguments of Isaksen~\cite{Isaksen-ProSpectra} and
Fausk-Isaksen~\cite{FauskIsaksen} used there.

A classic $p$-cyclotomic spectrum $X$ comes with two maps $X^{C_{p^{k}}}\to
X^{C_{p^{k-1}}}$: 
\begin{enumerate}
\item[$F$:] the canonical inclusion of fixed points, and
\item[$R$:] the composite 
\[
X^{C_{p^{k}}}\iso (\rho^{*}_{p}X^{C_{p}})^{C_{p^{k-1}}}\to 
(\rho^{*}_{p}X^{\Phi C_{p}})^{C_{p^{k-1}}}\overto{\simeq}
X^{C_{p^{k-1}}}
\]
\end{enumerate}
induced by the canonical map from the fixed points to geometric fixed
points and the $p$-cyclotomic structure map. 

\begin{defn}[Classic $TC$]\label{defn:classicTC}
For a classic $p$-cyclotomic spectrum $X$ and $k>0$, let
\[
TC^{k}(X;p)=\hoeq\bigl(R,F\colon X^{C_{p^{k}}}\rightrightarrows X^{C_{p^{k-1}}}\bigr),
\]
the homotopy equalizer of $R,F$.  Let 
\[
TC(X;p)=\holim_{k} TC^{k}(X;p).
\]
\end{defn}

Goodwillie~\cite[12.1]{Goodwillie-MSRI} first proposed an integral
form of $TC$ for cyclotomic spectra.  It is easier to take its
fundamental property~\cite[14.2]{Goodwillie-MSRI} and turn it into a
definition as in
Dundas-Goodwillie-McCarthy~\cite[VI.3.3.1]{DundasGoodwillieMcCarthy}:
define $TC(X)$ as the homotopy pullback 
\[
\xymatrix{%
TC(X)\ar[r]\ar[d]&X^{h\bT}\ar[d]\\
\prod_{p}TC(X;p)\phat\ar[r]
&\prod_{p} (\holim_{k}X^{hC_{p^{k}}})\phat
}
\]
where the map $TC(X;p)\to X^{hC_{p^{k}}}$ is the composite
\[
TC(X;p)\to TC^{k}(X;p)\to X^{C_{p^{k}}}\to X^{hC_{p^{k}}}.
\]
We note that the map $X^{h\bT}\to \holim_{k} X^{hC_{p^{k}}}$ becomes a
weak equivalence after $p$-completion (see, for
example,~\cite[VI.2.1.1]{DundasGoodwillieMcCarthy}); this immediately
implies the following well-known result.

\begin{prop}[Goodwillie~{\cite[14.1.(ii)]{Goodwillie-MSRI}}]
\label{prop:classicp}
Let $X$ be a classic cyclotomic spectrum.  The map $TC(X)\phat\to
TC(X;p)\phat$ is a weak equivalence.
\end{prop}

Because of this proposition, we can use $TC(X;p)$ exclusively in the
classic context and deduce results for $TC(X)\phat$ from the
equivalence.

Nikolaus-Scholze~\cite[II.1.6]{NikolausScholze} redefine cyclotomic
spectra and $p$-cyclotomic spectra in 
terms of Borel equivariant spectra: for a topological group $G$, the
homotopy theory or $\infty$-category of Borel $G$-spectra is
represented by the (relative) category of (left) $G$-objects in orthogonal
spectra, where we take the weak equivalences to be maps that are weak
equivalences of the underlying spectra. The
underlying object of a cyclotomic spectrum is a Borel $\bT$-spectrum.
The underlying object of a $p$-cyclotomic spectrum is a 
Borel $C_{p^{\infty}}$-spectrum, where $C_{p^{\infty}}=\colim_{k}
C_{p^{k}}$.  The structure map for a $p$-cyclotomic spectrum $X$ is a
map of Borel $C_{p^{\infty}}$-spectra
\[
\phi \colon X\to \rho^{*}_{p}(X^{tC_{p}})
\]
where $tC_{p}$ denotes the Tate fixed point spectrum (the fixed points
of the Tate spectrum of~\cite[p.~3]{GreenleesMay-Tate}) and $\rho_{p}$
is (as above) the $p$th root isomorphism $\bT\iso \bT/C_{p}$.  A
cyclotomic spectrum has such a structure map $\phi_{p}$ for each $p$,
which is required to be a map of Borel $\bT$-spectra, and we write
$\phi$ for the induced map into the product
\[
\phi \colon X\to \prod_{p}\rho^{*}_{p}(X^{tC_{p}}).
\]
Precisely, the $C_{p}$-Tate fixed points is a functor
from Borel $\bT$-spectra to Borel $\bT/C_{p}$-spectra or from Borel
$C_{p^{\infty}}$-spectra to Borel $C_{p^{\infty}}/C_{p}$-spectra and
the composite with $\rho^{*}_{p}$ then gives an endofunctor of Borel
$\bT$-spectra or Borel $C_{p}^{\infty}$-spectra.  The
$\infty$-category of $p$-cyclotomic spectra is the lax
equalizer (in the sense of~\cite[II.1.4]{NikolausScholze}) from the identity to the
functor $\rho^{*}_{p}(-)^{tC_{p}}$ and
the $\infty$-category of cyclotomic spectra is the lax equalizer from
the identity to the product (over primes $p$) of the functors
$\rho^{*}_{p}(-)^{tC_{p}}$ \cite[II.1.6]{NikolausScholze}.

For a Borel $\bT$-spectrum or Borel $C_{p^{\infty}}$-spectrum $X$, we
have a canonical map of Borel $\bT/C_{p}$-spectra or Borel
$C_{p^{\infty}}/C_{p}$-spectra from the $C_{p}$ homotopy fixed
points to the $C_{p}$-Tate fixed points
\[
X^{hC_{p}}\to X^{tC_{p}}
\]
and a canonical isomorphism 
\[
X^{h\bT}\iso (\rho^{*}_{p}(X^{hC_{p}}))^{h\bT}\quad \text{or}\quad 
X^{hC_{p^{\infty}}}\iso (\rho^{*}_{p}(X^{hC_{p}}))^{hC_{p^{\infty}}}.
\]
We then obtain a canonical map
\[
\can\colon X^{h\bT}\to \prod_{p} (\rho^{*}_{p}(X^{tC_{p}}))^{h\bT}\quad \text{or}\quad 
\can\colon X^{hC_{p^{\infty}}}\to (\rho^{*}_{p}(X^{tC_{p}}))^{hC_{p^{\infty}}}.
\]
In the case when $X$ is a cyclotomic spectrum or $p$-cyclotomic
spectrum, we have another map with the same domain and codomain
obtained by taking the homotopy fixed points of $\phi$
\[
\phi^{h\bT}\colon X^{h\bT}\to \prod_{p} (\rho^{*}_{p}(X^{tC_{p}}))^{h\bT}\quad \text{or}\quad 
\phi^{hC_{p^{\infty}}}\colon X^{hC_{p^{\infty}}}\to (\rho^{*}_{p}(X^{tC_{p}}))^{hC_{p^{\infty}}}.
\]
Following~\cite{BM-cycl},
Nikolaus-Scholze~\cite[II.1.8]{NikolausScholze} define $TC(X)$ and
$TC(X;p)$ as mapping spectra but show in
\cite[II.1.9]{NikolausScholze} that they may be computed as the
homotopy equalizer of $\can$ and $\phi^{h\bT}$ or $\can$ and
$\phi^{hC_{p^{\infty}}}$; we use this theorem as a definition.

\begin{defn}[Nikolaus-Scholze $TC$ {\cite[II.1.9]{NikolausScholze}}]
For a Nikolaus-Scholze cyclotomic spectrum $X$, define 
\[
TC(X)=\hoeq \bigl(\can,\phi^{h\bT}\colon X^{h\bT}\rightrightarrows \prod_{p} (\rho^{*}_{p}(X^{tC_{p}}))^{h\bT} \bigr).
\]
For a Nikolaus-Scholze $p$-cyclotomic spectrum $X$, define 
\[
TC(X;p)=\hoeq \bigl(\can,\phi^{hC_{p^{\infty}}}\colon X^{hC_{p^{\infty}}}\rightrightarrows (\rho^{*}_{p}(X^{tC_{p}}))^{hC_{p^{\infty}}} \bigr).
\]
\end{defn}

We have a forgetful functor from cyclotomic spectra to $p$-cyclotomic
spectra and projection onto the $p$ factor induces a map $TC(X)\to
TC(X;p)$.  The following proposition is well-known and implicit
in~\cite{NikolausScholze}. 

\begin{prop}
\label{prop:NSp}
Let $X$ be a Nikolaus-Scholze cyclotomic spectrum. If $X$ is either
bounded below or $p$-local, then the map $TC(X)\phat\to TC(X;p)\phat$ is
a weak equivalence.
\end{prop}

\begin{proof}
Let $\ell\neq p$ be prime.  If $X$ is $p$-local, then multiplication
by $\ell$ is an isomorphism on $X$; the canonical maps
\[
(\pi_{*}X)/C_{\ell}\to \pi_{*}(X_{hC_{\ell}}) \quad \text{and}\quad
\pi_{*}(X^{hC_{\ell}})\to (\pi_{*}(X))^{C_{\ell}}
\]
are then isomorphisms and $X^{tC_{\ell}}$ is trivial.   If $X$
is bounded below, then 
by~\cite[I.2.9]{NikolausScholze}, $X^{tC_{\ell}}$ is $\ell$-complete,
$(\rho^{*}_{\ell}(X^{tC_{\ell}}))^{h\bT}$ is $\ell$-complete, and the
$p$-completion of $(\rho^{*}_{\ell}(X^{tC_{\ell}}))^{h\bT}$ is trivial.
Thus, if $X$ is either $p$-local or bounded below, the map
\[
TC(X)\to \hoeq\bigl(\can,\phi_{p}^{h\bT}\colon
X^{h\bT}\rightrightarrows (\rho^{*}_{p}X^{tC_{p}})^{h\bT}\bigr)
\]
is a $p$-equivalence. The maps 
\[
X^{h\bT}\to X^{hC_{p^{\infty}}}\quad\text{and}\quad
(\rho^{*}_{p}X^{tC_{p}})^{h\bT}\to (\rho^{*}_{p}X^{tC_{p}})^{hC_{p^{\infty}}}
\]
are also $p$-equivalences (for example, by
\cite[VI.2.1.1]{DundasGoodwillieMcCarthy}) and it follows that the
map $TC(X)\phat\to TC(X;p)\phat$ is a weak equivalence.
\end{proof}

In particular, if all cyclotomic spectra under consideration are
either connective or $p$-local, as is the case in the statement of
Theorem~\ref{main}, results on $TC(-;p)\phat$ automatically imply the
corresponding results on $TC(-)\phat$.

The Nikolaus-Scholze setting also has an analogue of $TC^{k}(X;p)$:
let 
\[
TC^{k}(X;p)=
\hoeq\bigl(\can,\phi^{hC_{p^{k-1}}}\colon X^{hC_{p^{k}}}
\rightrightarrows (X^{tC_{p}})^{hC_{p^{k-1}}}\bigr).
\]
Since $(-)^{hC_{p^{\infty}}}\iso \holim (-)^{hC_{p^{k}}}$, we then
have an equivalence
\[
TC(X;p)\iso \holim_{k} TC^{k}(X;p).
\]
Nikolaus-Scholze~\cite[II.4.10]{NikolausScholze} (and its proof) show
that when $X$ is bounded below, we have a weak equivalence from
classic $TC^{k}(X;p)$ to Nikolaus-Scholze $TC^{k}(X;p)$, which in the
homotopy limit gives a weak equivalence from classic $TC(X;p)$ to
Nikolaus-Scholze $TC(X;p)$.

Finally, for a ring spectrum $R$, we get a cyclotomic spectrum
$THH(R)$, first constructed by B\"okstedt~\cite{Bokstedt}.  Now
several equivalent constructions exist~\cite{ABGHLM},
\cite[III]{NikolausScholze}; see~\cite[\S1]{DMPSW} for a discussion
and comparison.  We caution that it is standard abuse of notation to
write $TC(R)$ and $TC(R;p)$ for $TC(THH(R))$ and $TC(THH(R);p)$, and
this usually causes no confusion.

We need the fact from Patchkoria-Sagave~\cite[3.8]{PatchkoriaSagave}
that as a Borel $\bT$-spectrum $THH(R)$ is weakly equivalent to the
(left derived) cyclic bar construction $N^{cy}(R)$, constructed out of
(derived) smash powers of $R$.  It is a fact inherited from the smash
power that when $R$ is connective, so is $THH(R)$ and also that the
natural map 
\[
THH(R)\phat \to THH(R\phat)\phat
\]
is a weak equivalence of Borel $\bT$-spectra, and \textit{a fortiori}
(see~\cite[5.5]{BM-cycl}) a weak equivalence of both classic and
Nikolaus-Scholze cyclotomic spectra.

It is well-known to experts (see, for example, \cite[Add.~6.2]{HM2}
or~\cite[6.10]{BM-cycl}) that the $p$-completion of $TC(X;p)$ only
depends on the $p$-completion of $X$. In the classic case,
this is because $p$-completion commutes with fixed points.  In the
Nikolaus-Scholze case, this is because $p$-completion commutes with
homotopy fixed points and the natural map
\[
(X^{tC_{p}})\phat\to ((X\phat)^{tC_{p}})\phat
\]
is a weak equivalence for any Borel $C_{p}$-spectrum.  We summarize in
the following proposition.

\begin{prop}\label{prop:pcompl}
In either the classic or Nikolaus-Scholze setting, for any
$p$-cyclotomic spectrum $X$, the natural map
\[
TC(X;p)\phat\to TC(X\phat;p)\phat
\]
is a weak equivalence and for any ring spectrum $R$, the natural map
\[
TC(R;p)\phat\to TC(R\phat;p)\phat
\]
is a weak equivalence.
\end{prop}

For the purposes of our main results, we have proved the following
theorem.

\begin{thm}\label{thm:TCS}
In either the classic or the Nikolaus-Scholze setting, all pictured natural
maps 
\[
\xymatrix@C-1.5pc@R-1.25pc{%
TC(\bS)\phat\ar[r]\ar[d]&TC(\bS_{(p)})\phat\ar[r]\ar[d]&TC(\bS\phat)\phat\ar[d]\\
TC(\bS;p)\phat\ar[r]&TC(\bS_{(p)};p)\phat\ar[r]&TC(\bS\phat;p)\phat\\
TC(\tau_{\geq 0}L_{n}(\bS))\phat\ar[r]\ar[d]&TC(\tau_{\geq 0}L_{n}(\bS_{(p)}))\phat\ar[r]\ar[d]&TC(\tau_{\geq 0}L_{n}(\bS\phat))\phat\ar[d]\\
TC(\tau_{\geq 0}L_{n}(\bS);p)\phat\ar[r]&TC(\tau_{\geq 0}L_{n}(\bS_{(p)});p)\phat\ar[r]&TC(\tau_{\geq 0}L_{n}(\bS\phat);p)\phat\\
TC(L_{n}(\bS))\phat\ar[r]\ar[d]&TC(L_{n}(\bS_{(p)}))\phat\ar[r]\ar[d]&TC(L_{n}(\bS\phat))\phat\ar[d]\\
TC(L_{n}(\bS);p)\phat\ar[r]&TC(L_{n}(\bS_{(p)});p)\phat\ar[r]&TC(L_{n}(\bS\phat);p)\phat
}
\]
are weak equivalences.
\end{thm}

\section{The Hopkins-Ravenel Chromatic Convergence Theorem and $TC(\bS_{(p)})$}
\label{sec:conn}

In this section, we apply the Hopkins-Ravenel Chromatic Convergence
Theorem~\cite[7.5.7]{Ravenel-Nilpotence} to prove
Theorem~\ref{main:conn}. Specifically, Hopkins and Ravenel prove that
the canonical map 
\[
\bS_{(p)}\to \holim_{n} L_{n}\bS
\]
is a weak equivalence; indeed (it is well-known that) the argument
shows that the map induces a pro-isomorphism on $\pi_{q}$ for each
$q$. The proof of Theorem~\ref{main:conn} is to propagate the
pro-isomorphism on homotopy groups through the $TC$ construction.  We
prove the following theorem.

\begin{thm}\label{thm:mainconn}
The map of pro-spectra $\{TC^{k}(\bS_{(p)};p)\}\to \{TC^{k}(\tau_{\geq
0}L_{n}\bS;p)\}$ induces a pro-isomorphism on $\pi_{q}$ pro-groups for all $q\in \bZ$.
\end{thm}

Taking the homotopy limit
\begin{multline*}
TC(\bS_{(p)};p)=\holim_{k} TC^{k}(\bS_{(p)};p) \overto{\simeq} \holim_{k}\holim_{n} TC^{k}(\tau_{\geq 0}L_{n}\bS;p)\\
\overto{\simeq} \holim_{n}\holim_{k} TC^{k}(\tau_{\geq 0}L_{n}\bS;p)=
\holim_{n} TC(\tau_{\geq 0}L_{n}\bS;p),
\end{multline*}
we then get an $\bS_{(p)}$ version of
Theorem~\ref{main:conn}, which by Theorem~\ref{thm:TCS} is equivalent to
Theorem~\ref{main:conn}.  Because $\bS_{(p)}$ and $\tau_{\geq
0}L_{n}\bS$ are all connective spectra, both flavors of $TC$ are
equivalent, and we only need to prove the theorem for classic $TC$; we
discuss only this case.

The propagation argument uses results of
Isaksen~\cite{Isaksen-ProSpectra} and
Fausk-Isaksen~\cite{FauskIsaksen} on the homotopy theory of
pro-spectra. We emphasize that this
theory takes place in the point-set category of pro-objects in a
point-set category of spectra, in our case, the category of orthogonal
spectra indexed on $\{\bR^{n}\}$.  In particular,
Theorem~\ref{thm:mainconn} refers to a fixed (but unspecified)
point-set tower of functors $TC^{k}(-;p)$, point-set models $\tau_{\geq 0}L_{n}\bS$,
and system of point-set maps $\bS_{(p)}\to \tau_{\geq 0}L_{n}\bS$. As
clearly such a setup exists and the details play no role in
the arguments, we omit a discussion and precise specification.  On the
other hand, making certain specifications avoids some awkward
circumlocutions in statements, so we consider only models that satisfy
the following hypothesis.

\begin{conv}\label{conv:models}
The specific models used for the spectra $\bS_{(p)}$ and $\tau_{\geq
0}L_{n}\bS$ in this section are cofibrant in one of the standard (stable) model
structure on associative ring orthogonal spectra or 
commutative ring orthogonal spectra. 
\end{conv}

We recall that a property of a spectrum or map of spectra is said to
hold \term{levelwise} for a pro-spectrum or a level map of pro-spectra
if the constituent spectra or maps of spectra have that property.
More generally, the property is said to hold \term{essentially
levelwise} on a pro-spectrum or (general) map of pro-spectra if the
pro-spectrum or map is isomorphic in the pro-category to one where the
property holds levelwise.  In particular, a pro-spectrum is
\term{levelwise $m$-connected} if its constituent spectra are
$m$-connected and a level map of pro-spectra $X\to Y$ is a
\term{levelwise $m$-equivalence} if each map in the system is an
$m$-equivalence.  A map of pro-spectra $X\to Y$ is an
\term{essentially levelwise $m$-equivalence} if there exist
pro-isomorphisms $X'\iso X$ and $Y\iso Y'$ such that the composite
$X'\to Y'$ is represented by a level map that is a levelwise
$m$-equivalence.  The essentially levelwise $m$-equivalences plays a
central role in the theory of~\cite{Isaksen-ProSpectra,FauskIsaksen}.

\begin{defn}\label{defn:weakeq}
A map of pro-spectra is a \term{weak equivalence} (called
$\pi_{*}$-weak equivalence in \cite{Isaksen-ProSpectra} and
$\oH_{*}$-weak equivalence in \cite{FauskIsaksen}) if it is an
essentially levelwise $m$-equivalence for all $m\in \bZ$.  
\end{defn}

The main result on pro-spectra we need is the following adaptation
of~\cite[8.4]{Isaksen-ProSpectra} or \cite[9.13]{FauskIsaksen}.

\begin{thm}[Isaksen~{\cite[8.4]{Isaksen-ProSpectra}}]
\label{thm:isaksen}
Let $X=\{X_{s}\}$ and $Y=\{Y_{t}\}$ be levelwise $(-N)$-connected
pro-spectra for some $N\in \bZ$.  A map of pro-spectra $X\to Y$ is a weak
equivalence in the sense of Definition~\ref{defn:weakeq} if and only
if it induces on $\pi_{q}$ a pro-isomorphism of pro-groups for all $q>-N$.
\end{thm}

Although not stated explicitly, the proof of the theorem gives slightly
more information: the pro-isomorphic spectra $X'\iso X$ and $Y\iso Y'$
for which the composite map $X'\to Y'$ is represented by a levelwise
$m$-equivalence can be chosen so that the constituent spectra are also
$(-N)$-connected.  We get the following immediate corollary.

\begin{cor}\label{cor:isaksen}
For each $m\geq 0$, there exists levelwise connective pro-spectra
$X^{m}=\{X^{m}_{t}\}$ and $Y^{m}=\{Y^{m}_{t}\}$, pro-isomorphisms
$X^{m}\iso \{\bS_{(p)}\}$ and $Y^{m}\iso \{\tau_{\geq 0}L_{n}\bS\}$,
and a levelwise $m$-equivalence $X^{m}\to Y^{m}$ such that the diagram
\[
\xymatrix@-1pc{%
X^{m}\ar[r]\ar[d]_{\iso}&Y^{m}\ar[d]^{\iso}\\
\{\bS_{(p)}\}\ar[r]&\{\tau_{\geq 0}L_{n}\bS\}
}
\]
commutes.
\end{cor}

We now begin to apply this to get comparison results for the map
$\{\bS_{(p)}\}\to \{\tau_{\geq 0}L_{n}\bS\}$.  In the following
proposition, $(-)^{(k)}$ denotes $k$-fold smash power.

\begin{prop}
The map of pro-spectra $\bS_{(p)}\to \{n\mapsto (\tau_{\geq 0}L_{n}\bS)^{(k)}\}$
is a weak equivalence for each $k\geq 1$.
\end{prop}

\begin{proof}
Let $X^{m}$, $Y^{m}$ be as in Corollary~\ref{cor:isaksen}.  Writing
$(-)^{\dL(k)}$ for the derived smash power (in the stable category),
for each $t$ in the common indexing category for $X^{m},Y^{m}$, the map 
\[
(X^{m}_{t})^{\dL(k)}\to (Y^{m}_{t})^{\dL(k)}
\]
is an $m$-equivalence.  These maps assemble into a level map of pro-objects
in the stable category that fits into the following commutative diagram of maps of
pro-objects in the stable category. 
\[
\xymatrix@-1pc{%
(X^{m})^{\dL(k)}\ar[d]\ar[r]&(Y^{m})^{\dL(k)}\ar[d]\\
\{\bS_{(p)}^{(k)}\}\ar[r]&\{(\tau_{\geq 0}L_{n}\bS)^{(k)}\}
}
\]
The vertical maps are pro-isomorphisms of pro-objects in the stable
category and in particular induce pro-isomorphisms on homotopy
groups.  The top map is a levelwise isomorphism on $\pi_{q}$ for
$q<m$, and so the bottom map is a pro-isomorphism on $\pi_{q}$ for
$q<m$.  Since $m$ was arbitrary, the statement follows from
Theorem~\ref{thm:isaksen}. 
\end{proof}

In the following proposition, we use $N^{cy}$ to denote the cyclic bar
construction for ring spectra.

\begin{prop}\label{prop:Ncy}
The map of pro-spectra $\{N^{cy}(\bS_{(p)})\}\to \{N^{cy}(\tau_{\geq
0}L_{n}\bS)\}$ is a weak equivalence.
\end{prop}

\begin{proof}
The cyclic bar construction $N^{cy}(R)$ is the geometric realization of a
simplicial spectrum with $q$th simplicial level the $(q+1)$th smash
power $R^{(q+1)}$.  The hypotheses on $\bS_{(p)}$ and $\tau_{\geq
0}L_{n}\bS$ (Convention~\ref{conv:models}) imply that the geometric
realization is equivalent to the thickened realization obtained by
gluing just using the face maps.  The thickened realization $N^{cy}_{\Delta}(R)$ has a
filtration by Hurewicz cofibrations 
\[
*=N^{cy}_{\Delta}(R)_{-1}\to N^{cy}_{\Delta}(R)_{0}\to N^{cy}_{\Delta}(R)_{1}\to N^{cy}_{\Delta}(R)_{2} \to \dotsb
\]
with $N^{cy}_{\Delta}(R)=\colim N^{cy}_{\Delta}(R)_{m}$ and $m$th
filtration quotient $R^{(m+1)}\sma \Delta[n]/\partial$.  
Because the spectra $\bS_{(p)}^{(k)}$ and $(\tau_{\geq
0}L_{n}\bS)^{(k)}$ are connective for all $k$, the inclusion of the 
$m$th filtration piece $N^{cy}_{\Delta}(R)_{m}\to N^{cy}_{\Delta}(R)$
is therefore an $m$-equivalence for these ring spectra.  In
particular, to prove the proposition, it suffices to prove that for
each $m$, the map of pro-spectra
\[
\{N^{cy}_{\Delta}(\bS_{(p)})_{m}\}\to
\{N^{cy}_{\Delta}(\tau_{\geq 0}L_{n}\bS)_{m}\}
\]
induces a pro-isomorphism on each $\pi_{q}$ pro-group.  This
certainly holds for $m=-1$ where both sides are trivial and holds by
induction for all $m$ by the previous proposition.
\end{proof}

The cyclic bar construction $N^{cy}(R)$ comes with a natural action of
the Lie group $\bT$.  We can then form the homotopy orbit
spectrum $N^{cy}(R)_{hC_{q}}=(N^{cy}(R)\sma EC_{q+})/C_{q}$.  By
naturality, we get maps of spectra $N^{cy}(\bS_{(p)})_{hC_{q}}\to
N^{cy}(\tau_{\geq 0}L_{n}\bS)_{hC_{q}}$ and a map of pro-spectra
\begin{equation}\label{eq:horbits}
\{N^{cy}(\bS_{(p)})_{hC_{q}}\}\to 
\{N^{cy}(\tau_{\geq 0}L_{n}\bS)_{hC_{q}}\}.
\end{equation}

\begin{prop}\label{prop:horbits}
For each $q\in \bN$, the map of pro-spectra~\eqref{eq:horbits} is a
weak equivalence. 
\end{prop}

\begin{proof}
Theorem~\ref{thm:isaksen} works just as written in the Borel 
equivariant context where we understand weak equivalences (and
$m$-equivalences) as maps that are weak equivalences (and
$m$-equivalences, resp.) of the underlying spectra.  In terms
of~\cite{FauskIsaksen}, we look at the model structure on the category
of orthogonal spectra with $C_{q}$-actions where the fibrations and
weak equivalences are the equivariant maps that are fibrations and
weak equivalences of the underlying spectra and use the Postnikov
section $t$-structure.  (This is the $\aF$-model structure
of~\cite[IV.6.5]{MM} on orthogonal spectra indexed on $\{\bR^{n}\}$ where
$\aF$ is the family containing only the trivial subgroup.)  By the
previous proposition and the Borel equivariant version of
Theorem~\ref{thm:isaksen}, for each $m$ we can find pro-isomorphisms
\[
X\iso N^{cy}(\bS_{(p)}), \qquad 
N^{cy}(\tau_{\geq 0}L_{n}\bS)\iso Y
\]
in the pro-category of orthogonal $C_{q}$-spectra
so that the composite $X\to Y$ is represented by a level map that is a
levelwise $m$-equivalence.  We then have pro-isomorphisms
\[
X_{hC_{q}}\iso N^{cy}(\bS_{(p)})_{hC_{q}}, \qquad 
N^{cy}(\tau_{\geq 0}L_{n}\bS)_{hC_{q}}\iso Y_{hC_{q}}
\]
and the composite map $X_{hC_{q}}\to Y_{hC_{q}}$ remains a level map
and a levelwise $m$-equivalence. 
\end{proof}

For the next step, we need a model for $THH(\bS_{(p)})$ and the tower
$THH(\tau_{\geq 0}L_{n}\bS)$ with the correct genuine
$\bT$-equivariant homotopy type to construct classic $TC(-;p)$ and
$TC^{k}(-;p)$.  It is convenient to assume the models to be fibrant in
the category of $p$-cyclotomic spectra of~\cite{BM-cycl}: this ensures
that the point-set $C_{p^{k}}$-fixed points have the correct homotopy type for
all $k$.   With this hypothesis, we get the following comparison
result for the $C_{p^{k}}$-fixed points applied to the map of
pro-$\bT$-spectra $\{THH(\bS_{(p)})\}\to \{THH(\tau_{\geq
0}L_{n}\bS)\}$.

\begin{prop}
For each $k$, the map of pro-spectra 
\[
\{THH(\bS_{(p)})^{C_{p^{k}}}\}\to \{THH(\tau_{\geq
0}L_{n}\bS)^{C_{p^{k}}}\}
\]
is a weak equivalence.
\end{prop}

\begin{proof}
For $k=0$, the $C_{1}$-fixed points are the underlying spectrum and
the statement follows from the Proposition~\ref{prop:Ncy} using the Borel
equivariant weak equivalence of $THH(R)$ and $N^{cy}(R)$ (for ring
spectra $R$ which are cofibrant or are cofibrant
as commutative ring spectra).

Now let $k\geq 1$.  For a $p$-cyclotomic spectrum $X$, Theorem~2.2
of~\cite{HM2} constructs a ``fundamental cofiber sequence'' in the
stable category
\[
X_{hC_{p^{k}}}\to X^{C_{p^{k}}}\to X^{C_{p^{k-1}}},
\]
natural in maps of $p$-cyclotomic spectra.  Applying $\pi_{q}$ we get
a long exact sequence, and for the map of pro-spectra 
\[
\{THH(\bS_{(p)})\}\to \{THH(\tau_{\geq 0}L_{n}\bS)\},
\]
we get a homomorphism of long exact sequences in the pro-category of
abelian groups.  The map
\[
\{\pi_{q}(THH(\bS_{(p)})_{hC_{p^{k}}}\}\to 
\{\pi_{q}(THH(\tau_{\geq 0}L_{n}\bS)_{hC_{p^{k}}}\}
\]
is a pro-isomorphism by Proposition~\ref{prop:horbits} and the map 
\[
\{\pi_{q}(THH(\bS_{(p)})^{C_{p^{k-1}}}\}\to 
\{\pi_{q}(THH(\tau_{\geq 0}L_{n}\bS)^{C_{p^{k-1}}}\}
\]
is a pro-isomorphism by induction, and we conclude that the map
\[
\{\pi_{q}(THH(\bS_{(p)})^{C_{p^{k}}}\}\to 
\{\pi_{q}(THH(\tau_{\geq 0}L_{n}\bS)^{C_{p^{k}}}\}
\]
is a pro-isomorphism by the five lemma.
\end{proof}

Because $TC^{k}(-;p)$ is defined as a homotopy equalizer of a pair of
maps between the fixed points, we immediately deduce the following
proposition. 

\begin{prop}\label{prop:tck}
For each $k$, the map of pro-spectra $\{TC^{k}(\bS_{(p)})\}\to
\{TC^{k}(\tau_{\geq 0}L_{n}\bS)\}$ is a weak equivalence.
\end{prop}

Theorem~\ref{thm:mainconn} follows easily from the preceding
proposition.

\section{Barthel's chromatic convergence criterion and $TC(\bS_{(p)})$}
\label{sec:nonconn}\label{sec:cc}

In this section, we prove Theorems~\ref{main} and~\ref{main:cc} using Barthel's
chromatic convergence criterion.  This criterion abstracts the Hopkins-Ravenel
proof of the Chromatic Convergence Theorem and states it in general
terms for general $p$-local spectra as follows.

\begin{thm}[Barthel~{\cite[3.8]{Barthel-ChromaticConvergence}}]
Let $X$ be a $p$-local spectrum whose $BP$-homology $BP_{*}X$ has
finite projective dimension as a graded $BP_{*}$-module.  Then $X\to
\holim L_{n}X$ is a weak equivalence.
\end{thm}

We need only the following special cases.

\begin{cor}\label{cor:Bp}
Let $G=\bT$ or $G=C_{p^{k}}$ for some $k\geq 0$; then the map
$(\Sigma^{\infty}_{+}BG)_{(p)}\to  \holim_{n}L_{n}(\Sigma^{\infty}_{+}BG)$
is a weak equivalence.
\end{cor}

\begin{proof}
The Atiyah-Hirzebruch spectral sequence identifies $BP_{*}(B\bT)$ as
a free graded $BP_{*}$-module and it therefore has projective dimension~0.
For $G=C_{p^{k}}$, the case $k=0$ is the Hopkins-Ravenel Chromatic
Convergence Theorem. For $k>0$, the
Johnson-Wilson~\cite[(2.11)]{JohnsonWilson-BPpgroups} argument generalizes from
$C_{p}$ to $C_{p^{k}}$ to show that $BP_{*}(BC_{p^{k}})$ is projective
dimension~1 as a graded $BP_{*}$-module.
\end{proof}

The previous corollary immediately proves Theorem~\ref{main:cc}.

\begin{proof}[Proof of Theorem~\ref{main:cc}]
The main result of~\cite{BHM} (in the case of the trivial group)
identifies $TC(\bS\phat;p)\phat\simeq TC(\bS;p)\phat$ as 
\[
\bS\phat \vee \Fib(\text{tr}\colon \Sigma \Sigma^{\infty}_{+}B\bT\to \bS)\phat
\]
where $\text{tr}$ denotes the transfer.  Commuting wedges, suspension,
and fiber sequences with homotopy limits, the result follows from
Corollary~\ref{cor:Bp}. 
\end{proof}

The remainder of the section proves the following theorem.

\begin{thm}\label{thm:main}
The map $TC(\bS_{(p)};p)\to \holim TC(L_{n}\bS;p)$
is a weak equivalence.
\end{thm}

The theorem is meant to apply to both classic $TC$ and
Nikolaus-Scholze $TC$, which for the non-connective spectra $L_{n}\bS$
are not equivalent.  In both cases, by Theorem~\ref{thm:TCS} the
result for $\bS_{(p)}$ implies the result for $\bS$ and $\bS\phat$
after $p$-completion and in particular implies Theorem~\ref{main}.

For both classic $TC$ and Nikolaus-Scholze $TC$, the argument starts
with the observation that because $L_{n}$-localization is smashing,
$THH(L_{n}\bS)$ is $L_{n}$-local.  It
follows that the unit map $\bS\to THH(L_{n}\bS)$ factors through a
weak equivalence $L_{n}\bS\to THH(L_{n}\bS)$.  Since the unit map
is $\bT$-equivariant, we get a weak equivalence $L_{n}\bS\to
THH(L_{n}\bS)$ of Borel $\bT$-spectra for the trivial $\bT$-action on
$L_{n}\bS$.  Taking homotopy orbits, we obtain the following
proposition. 

\begin{prop}
The map $\bS\to THH(L_{n}\bS)$ induces a weak equivalence
\[
L_{n}(\Sigma^{\infty}_{+}BC_{p^{k}})\to THH(L_{n}\bS)_{hC_{p^{k}}}.
\]
\end{prop}

This proposition combined with Corollary~\ref{cor:Bp} then implies the
following proposition.

\begin{prop}\label{prop:nchorbits}
The map $THH(\bS_{(p)})_{hC_{p^{k}}}\to \holim_{n}
THH(L_{n}\bS)_{hC_{p^{k}}}$ is a weak equivalence.
\end{prop}

Because homotopy fixed points commute with homotopy limits, the map
\[
THH(\bS_{(p)})^{hC_{p^{\infty}}}\to \holim_{n} THH(L_{n}\bS)^{hC_{p^{\infty}}}
\]
is also a weak equivalence.  Applying this observation and the
previous proposition in the case $k=1$, we see that the map
\[
THH(\bS_{(p)})^{tC_{p}}\to \holim_{n} THH(L_{n}\bS)^{tC_{p}}
\]
is a weak equivalence.  Commuting homotopy fixed points with homotopy
limits again, we see that the map
\[
(THH(\bS_{(p)})^{tC_{p}})^{h(C_{p^{\infty}}/C_{p})}\to \holim_{n} (THH(L_{n}\bS)^{tC_{p}})^{h(C_{p^{\infty}}/C_{p})}
\]
is a weak equivalence.  We can now prove Theorem~\ref{thm:main}.

\begin{proof}[Proof of Theorem~\ref{thm:main} for Nikolaus-Scholze
$TC$]
Using the notation of Section~\ref{sec:review}, we have shown above
that the maps 
\begin{gather*}
THH(\bS_{(p)})^{hC_{p^{\infty}}}\to 
\holim_{n} THH(L_{n}\bS)^{hC_{p^{\infty}}}\\
(\rho^{*}_{p}(THH(\bS_{(p)})^{tC_{p}}))^{hC_{p^{\infty}}}\to 
\holim_{n} (\rho^{*}_{p}(THH(L_{n}\bS)^{tC_{p}}))^{hC_{p^{\infty}}}.
\end{gather*}
are weak equivalences.  Since Nikolaus-Scholze $TC(-;p)$ is the fiber
of a map between these spectra, we get a weak equivalence 
\[
TC(\bS_{(p)};p)\to 
\holim_{n} TC(L_{n}\bS;p).\qedhere
\]
\end{proof}

\begin{proof}[Proof of Theorem~\ref{thm:main} for classic $TC$]
Recall that the ``fundamental cofiber sequence'' (Theorem~2.2
of~\cite{HM2})
\[
X_{hC_{p^{k}}}\to X^{C_{p^{k}}}\to X^{C_{p^{k-1}}}
\]
inductively relates the fixed points and homotopy orbits of a
$p$-cyclotomic spectrum.  By Proposition~\ref{prop:nchorbits} and
induction, we see that the map  
\[
THH(\bS_{p})^{C_{p^{k}}}\to \holim_{n} THH(L_{n}\bS)^{C_{p^{k}}}
\]
is a weak equivalence for all $k$ and we conclude that the map
\[
\holim_{k}
THH(\bS_{p})^{C_{p^{k}}}\to \holim_{k,n} THH(L_{n}\bS)^{C_{p^{k}}}
\]
is a weak equivalence.  Since $TC(-;p)$ is the homotopy fiber of a
self-map of $\holim_{k}(-)^{C_{k}}$, we get a weak equivalence 
\[
TC(\bS_{(p)};p)\to 
\holim_{n} TC(L_{n}\bS;p).\qedhere
\]
\end{proof}


\bibliographystyle{plain}
\bibliography{bluman}

\end{document}